\documentclass[a4paper,11pt]{article}
\usepackage[ngerman, english]{babel}
\usepackage[T1]{fontenc}
\usepackage[utf8]{inputenc}
\usepackage{amsmath}
\usepackage{amssymb}
\usepackage{amsthm}
\usepackage{graphicx}

\graphicspath{{figures/}{figures_new/}}

\makeatletter
\def\input@path{{figures/}{figures_new/}}
\makeatother

\usepackage{here}
\usepackage{color}
\usepackage{xcolor}
\usepackage{enumerate}
\usepackage{lmodern}
\usepackage{fancyvrb}
\usepackage[plainpages=false]{hyperref}
\usepackage{caption}
\usepackage{subfigure}
\usepackage{epstopdf}
\newtheorem{defi}{Definition}[section]
\newtheorem{theo}[defi]{Theorem}

\newtheorem{obs}[defi]{Observation}
\newtheorem{cor}[defi]{Corollary}
\newtheorem{prob}[defi]{Problem}

\newcounter{claimcount}
\newenvironment{claim}{\refstepcounter{claimcount}\textbf{Claim \arabic{claimcount}.}}{}

\usepackage{etoolbox}
\AtBeginEnvironment{proof}{\setcounter{claimcount}{0}}

\theoremstyle{remark}

\newcommand{\ENDproof}{\hfill $\blacksquare$\medskip\par}

\title{Fractional factors and component factors in graphs with isolated toughness smaller than 1}

\author{Isaak H.~Wolf \thanks{Funded by Deutsche Forschungsgemeinschaft (DFG) - 445863039} \\
\footnotesize
		Department of Mathematics, Paderborn University, Warburger Str.\ 100, 33098 Paderborn,
		Germany.
	\\
\\ \footnotesize isaak.wolf@uni-paderborn.de}

\date{}

\begin{document}

\maketitle

\begin{abstract}
Let $G$ be a simple graph and let $n,m$ be two integers with $0<m<n$. We prove that $iso(G-S)\leq \frac{n}{m}|S|$ for every $S \subset V(G)$ if and only if $G$ has a $\{C_{2i+1},T \colon 1 \leq i < \frac{m}{n-m}, T\in\mathcal{T}_{\frac{n}{m}}\}$-factor, where $iso(G-S)$ denotes the number of isolated vertices of $G-S$ and $\mathcal{T}_{\frac{n}{m}}$ is a special family of trees. Furthermore, we characterize the trees in $\mathcal{T}_{\frac{n}{m}}$ in terms of their bipartition.
\end{abstract}


\section{Introduction, notation and main result}
We consider finite graphs that may have parallel edges but no loops. A graph without parallel edges is called \emph{simple}. For a graph $G$, the vertex-set is denoted by $V(G)$ and the edge-set by $E(G)$. Let $X,Y \subseteq V(G)$ be two disjoint sets. The set of edges with one endvertex in $X$ and one in $Y$ is denoted by $E_G(X,Y)$. Let $v \in V(G)$. We write $\partial_G(v)$ for $E_G(\{v\},V(G)\setminus\{v\})$, and $d_G(v)$ for $|\partial_G(v)|$. The indices will be omitted, if there is no harm of confusion. The graph induced by $V(G)\setminus X$ is denoted by $G-X$. For $E \subseteq E(G)$, the graph obtained from $G$ by deleting every edge in $E$ is denoted by $G-E$; for convenience we write $G-e$ for $G-E$ if $E$ consists of a single edge $e$. Furthermore, we say that $E$ induces a subgraph $H$ of $G$ if $E(H) = E$ and $V(H)$ contains all vertices of $G$ that are incident with an edge of $E$.

The path with $i$ vertices is denoted by $P_i$. The complete bipartite graph with partition sizes $i,j$ (where $i \leq j$) is denoted by $K_{i,j}$; if $i=1$, then $K_{i,j}$ is a \emph{star}. A \emph{circuit} is a 2-regular connected graph; $C_i$ denotes the circuit with $i$ vertices. A connected graph without circuits is a \emph{tree}. For a tree $T$, every vertex of degree 1 is a \emph{leaf} of $T$; the set of leaves of $T$ is denoted by $Leaf(T)$. Moreover, every edge incident with a leaf is a \emph{pendant edge} of $T$.

If a subgraph of $G$ is spanning, then it is a \emph{factor} of $G$. Let $\mathcal{G}$ be a set of graphs. A factor $F$ of $G$ is a $\mathcal{G}$-factor, if every component of $F$ is isomorphic to an element of $\mathcal{G}$. Let $g_1,f_1:V(G) \to \mathbb{Z}$ and $g_2,f_2:V(G) \to \mathbb{R}$ be functions with $g_i(w)\leq f_i(w)$ for every $w \in V(G)$ and every $i \in \{1,2\}$. A factor $F$ of $G$ is a \emph{$(g_1,f_1)$-factor}, if $g_1(w) \leq d_F(w) \leq f_1(w)$ for every $w \in V(G)$. For a function $h: E(G) \to [0,1]$, we define $d^{h}(v):= \sum_{e \in \partial_{G}(v)}h(e)$. If $g_2(w)\leq d^h(w) \leq f_2(w)$ for every $w \in V(G)$, then $h$ is a \emph{fractional $(g_2,f_2)$-factor} of $G$. Additionally, if $g_2(w)=a$ and $f_2(w)=b$ for every $w \in V(G)$, then a fractional $(g_2,f_2)$-factor is called a \emph{fractional $[a,b]$-factor}. The set of isolated vertices of $G$ is denoted by $Iso(G)$; we write $iso(G)$ for $|Iso(G)|$.

The isolated toughness of a graph $G$, denoted by $I(G)$, was first introduced in \cite{yang2001} and is defined as follows:
$$I(G)=\mathrm{min}\left\{\frac{|S|}{iso(G-S)} \colon S \subseteq V(G), iso(G-S)\geq 2\right\}$$
if $G$ is not a complete graph and $I(G)=\infty$ otherwise. For $t \in \mathbb{R}$, a graph $G$ is \emph{isolated $t$-tough} if $I(G) \geq t$. The isolated toughness is strongly related to the existence of fractional factors and specific component factors. Tutte \cite{tutte19531} characterized isolated $1$-tough graphs by the existence of component factors as follows.

\begin{theo}[Tutte \cite{tutte19531}]
Let $G$ be a simple graph. Then, $G$ has a $\{K_{1,1},C_i \colon i \geq 3\}$-factor if and only if
$$iso(G-S)\leq |S| \quad \text{ for all } S \subset V(G).$$
\end{theo}

This result was extended by Amahashi, Kano \cite{amahashi1982factors} and Las Vergnas \cite{las1978extension} to isolated $\frac{1}{n}$-tough graphs.

\begin{theo}[Amahashi, Kano \cite{amahashi1982factors}, Las Vergnas \cite{las1978extension}]
\label{theo:star_factor}
Let $G$ be a simple graph and let $n \geq 2$ be an integer. Then, $G$ has a $\{K_{1,i} \colon 1\leq i\leq n\}$-factor if and only if
$$iso(G-S)\leq n|S| \quad \text{ for all } S \subset V(G).$$
\end{theo}

Kano, Lu and Yu \cite{kano2010component} asked for a general relation between isolated toughness and the existence of component factors.

\begin{prob}[Problem 1 in \cite{kano2010component}, Problem 7.10 in \cite{factors_and_factorizations_book}]\label{prob:I(G)_component_factors}
Let $G$ be a simple graph and let $n,m$ be two positive integers. If
$$iso(G-S)\leq \frac{n}{m} |S| \quad \text{ for all } \emptyset \neq S \subset V(G),$$
what factor does $G$ have?
\end{prob}

The same authors \cite{Yu2019FractionalFC} gave an answer to Problem~\ref{prob:I(G)_component_factors} when $n$ is odd, $n \geq 3$ and $m=2$. Let $\mathcal{T}(3)$ be the set of trees that can be obtained as follows (see \cite{Yu2019FractionalFC} for a more detailed definition):
\begin{enumerate}
\item start with a tree $T$ in which every vertex has degree $1$ or $3$,
\item insert a new vertex of degree $2$ into every edge of $T$,
\item add a new pendant edge to every leaf of $T$.
\end{enumerate}

For every integer $k\geq 2$, let $\mathcal{T}(2k+1)$ be the set of trees that can be obtained as follows (see \cite{Yu2019FractionalFC} for a more detailed definition):
\begin{enumerate}
\item start with a tree $T$ such that for every $v \in V(T)$
\begin{itemize}
\item $d_{T-Leaf(T)}(v) \in \{1,3, \ldots 2k+1\}$, and
\item $2|\{w \colon w \in Leaf(T) \cap N_T(v)\}|+d_{T-Leaf(T)}(v)\leq 2k+1$,
\end{itemize}
\item insert a new vertex of degree $2$ into every edge of $T-Leaf(T)$,
\item for every $v \in T-Leaf(T)$ with $d_{T-Leaf(T)}(v)=2l+1<2k+1$, add $k-l-|\{w \colon w \in Leaf(T) \cap N_T(v)\}|$ new pendant edges to $v$.
\end{enumerate}

\begin{theo}[Kano, Lu, Yu \cite{Yu2019FractionalFC}]
A simple graph $G$ has a $\{P_2,C_3,P_5,T \colon T \in \mathcal{T}(3)\}$-factor if and only if
$$iso(G-S)\leq \frac{3}{2}|S| \quad \text{ for all } S \subset V(G).$$
\end{theo}

\begin{theo}[Kano, Lu, Yu \cite{Yu2019FractionalFC}]
Let $k \geq 2$ be an integer. A simple graph $G$ has a $\{K_{1,i},T \colon 1 \leq i \leq k, T \in \mathcal{T}(2k+1)\}$-factor if and only if
$$iso(G-S)\leq \frac{2k+1}{2}|S| \quad \text{ for all } S \subset V(G).$$
\end{theo}

 We extend these results and give an answer to Problem~\ref{prob:I(G)_component_factors} when $n>m$. For every two integers $n,m$ with $0<m<n$ let $\mathcal{T}_{\frac{n}{m}}$ be the set of trees $T$ such that
\begin{itemize}
\item $iso(T-S) \leq \frac{n}{m} \vert S \vert$ for all $S\subset V(T)$, and
\item for every $e\in E(T)$ there is a set $S^* \subset V(T)$ with $iso((T-e)-S^*) > \frac{n}{m} \vert S^* \vert$.
\end{itemize}

The following theorem is the main result of this paper.

\begin{theo} \label{theo:equivalenz_iso_frac_component}
Let $G$ be a simple graph and let $n,m$ be integers with $0<m<n$. Then the following statements are equivalent:
\begin{itemize}
\item[$1)$] $iso(G-S)\leq \frac{n}{m}|S|$ for every $S \subset V(G)$.
\item[$2)$] $G$ has a fractional $[1,\frac{n}{m}]$-factor.
\item[$3)$] $G$ has a fractional $[1,\frac{n}{m}]$-factor with values in $\{0,\frac{1}{m},...,\frac{m-1}{m},1\}$.
\item[$4)$] $G$ has a $\{C_{2i+1},T \colon 1 \leq i < \frac{m}{n-m}, T\in\mathcal{T}_{\frac{n}{m}}\}$-factor.
\end{itemize}
\end{theo}

The paper is structured as follows. In Section~\ref{sec:I(G)_frac_fac} we give a relation between the isolated toughness and the existence of fractional factors, which proves the equivalence of $1)$, $2)$ and $3)$. In Section~\ref{sec:I(G)_comp_fac} we prove the equivalence of $1)$ and $4)$ by using fractional factors. In Section~\ref{sec:Tnm} we characterize the trees in $\mathcal{T}_{\frac{n}{m}}$ and deduce further structural properties.

\section{Isolated vertex conditions and fractional factors}\label{sec:I(G)_frac_fac}

There is a strong relation between the isolated toughness of a graph and the existence of fractional $[1,\frac{n}{m}]$-factors. When $\frac{n}{m}$ is an integer, Ma, Wang and Li \cite{ma2009isolated} obtained the following relation.

\begin{theo} [Ma, Wang, Li \cite{ma2009isolated}] \label{fractfactors0}
Let $G$ be a simple graph and $b > 1$ be an integer. Then
\begin{align*}
iso(G-S) \leq b \vert S \vert \quad \text{for all } S\subset V(G)
\end{align*}
 if and only if $G$ has a fractional $[1,b]$-factor.
\end{theo}

As shown by Yu, Kano and Lu \cite{Yu2019FractionalFC}, similar results are true for isolated $\frac{2}{n}$-tough graphs, where $n$ is an odd integer with $n \geq 3$.

\begin{theo} [Kano, Lu, Yu \cite{Yu2019FractionalFC}] \label{fractfactors1}
Let $G$ be a simple graph and $k\geq 1$ be an integer. Then
\begin{align*}
iso(G-S) \leq \frac{2k+1}{2} \vert S \vert \quad \text{for all } S\subset V(G)
\end{align*}
 if and only if $G$ has a fractional $[1,\frac{2k+1}{2}]$-factor with values in $\{0,\frac{1}{2},1\}$.
\end{theo}

It turned out that their proof also works for isolated $\frac{m}{n}$-tough graphs, where $n,m$ are arbitrary integers with $0<m<n$. By substituting $2k+1$ with $n$ and 2 with $m$ in the proof of Theorem \ref{fractfactors1} given in \cite{Yu2019FractionalFC}, this result can be extended as follows. 

\begin{theo} \label{fractfactors2}
Let $G$ be a simple graph and let $n,m$ be integers with $0<m<n$. Then
\begin{align} \label{isocondhilfstheo}
\tag{1}
iso(G-S) \leq \frac{n}{m} \vert S \vert \quad \text{for all } S\subset V(G)
\end{align}
if and only if $G$ has a fractional $[1,\frac{n}{m}]$-factor with values in $\{0,\frac{1}{m},...,\frac{m-1}{m},1\}$.
\end{theo}

For the sake of completeness, in the remainder of this section, we state the proof of \cite{Yu2019FractionalFC} (with the substitutions mentioned above) and deduce the equivalence of statements $1)$, $2)$ and $3)$ of Theorem~\ref{theo:equivalenz_iso_frac_component}.

The main tool to prove Theorem~\ref{fractfactors1} (respectively, Theorem~\ref{fractfactors2}) is provided by the next theorem. For a function $f:V(G) \to \mathbb{Z}^{+} \cup \{0\}$ and a set $X \subseteq V(G)$, set $f(X) := \sum_{x \in X} f(x)$.

\begin{theo} [Anstee \cite{ANSTEE199029}, Heinrich et al. \cite{heinrich_et_al_1990}] \label{gffactors}
Let $G$ be a graph and $g,f:V(G) \to \mathbb{Z}^{+} \cup \{0\}$ with $0 \leq g(x) < f(x)$ for all $x \in V(G)$. Then $G$ has a $(g,f)$-factor if and only if
\begin{align*}
g(T)-d_{G-S}(T) \leq f(S) \text{ for all } S \subset V(G),
\end{align*}
where $T=\{v \in V(G) \setminus S \colon d_{G-S}(v)<g(v)\}$.
\end{theo}

\begin{proof} [Proof of Theorem~\ref{fractfactors2} (cf. Kano, Lu, Yu \cite{Yu2019FractionalFC}).]
Assume that $G$ satisfies (\ref{isocondhilfstheo}). Let $G^{*}$ denote the graph obtained from $G$ by replacing each edge $e$ of $G$ by $m$ parallel edges $e(1),...,e(m)$. Then $V(G^{*})=V(G)$, and $d_{G^{*}}(v)=m \cdot d_{G}(v)$ for every $v \in V(G^{*})$. Define two functions $g,f: V(G^{*}) \to \mathbb{Z}^{+} \cup \{0\}$ as
\begin{align*}
g(x)=m \hspace{0.5cm} \text{ and } \hspace{0.5cm} f(x)=n \hspace{0.5cm} \text{ for all } x\in V(G^{*}).
\end{align*}
Then $g<f$, and for any $S \subset V(G^{*})$, we have
\begin{align*}
T&=\{v \in V(G^{*}) \setminus S \colon d_{G^{*}-S}(v)<g(v)=m\}\\
&=\{v \in V(G^{*}) \setminus S \colon d_{G^{*}-S}(v)=0\}\\
&=Iso(G-S).
\end{align*}
Thus it follows from the above equality and (\ref{isocondhilfstheo}) that
\begin{align*}
g(T)-d_{G^{*}-S}(T)&=m \cdot iso(G-S)-0\\
&\leq n \vert S \vert =f(S).
\end{align*}
Hence by Theorem \ref{gffactors}, $G^{*}$ has a $(g,f)$-factor $F$. Now we construct a fractional $[1,\frac{n}{m}]$-factor $h:E(G) \to \{0,\frac{1}{m},...,\frac{m-1}{m},1\}$ as follows: for every edge $e$ of $G$, define $h(e)=\frac{k(e)}{m}$ where $k(e)$ is the number of integers $i \in \{1,...,m\}$ with $e(i) \in E(F)$. It is easy to see that $h$ is the desired fractional $[1,\frac{n}{m}]$-factor with values in $\{0,\frac{1}{m},...,\frac{m-1}{m},1\}$.

Next assume that $G$ has a fractional $[1,\frac{n}{m}]$-factor $h$. Let $S \subset V(G)$, and let $F$ be the spanning subgraph of $G$ induced by $\{e \in E(G) \colon h(e) \neq 0\}$. Clearly, the neighbours of each isolated vertex $u$ of $G-S$ are contained in $S$ and $d^{h}(u) \geq 1$, thus we have
\begin{align*}
iso(G-S) &\leq \sum \limits_{e \in E_{F}(Iso(G-S),S)} h(e)\\
& \leq \sum \limits_{x \in S} d^{h}(x) \leq \frac{n}{m} \vert S \vert.
\end{align*}
Hence, $iso(G-S) \leq \frac{n}{m} \vert S \vert$, i.e. (\ref{isocondhilfstheo}) holds.
\end{proof}

The fact, $h$ has values in $\{0,\frac{1}{m},...,\frac{m-1}{m},1\}$, is not needed in the second part of the proof of Theorem \ref{fractfactors2}. As a consequence, we obtain the following corollary:

\begin{cor}
Let $G$ be a simple graph and let $n,m$ be integers with $0<m<n$. If $G$ has a fractional $[1,\frac{n}{m}]$-factor, then $G$ has a fractional $[1,\frac{n}{m}]$-factor with values in $\{0,\frac{1}{m},...,\frac{m-1}{m},1\}$.
\end{cor}

Therefore, the equivalence of statements $1)$, $2)$ and $3)$ of Theorem~\ref{theo:equivalenz_iso_frac_component} is proved.

\section{Isolated vertex conditions and component factors}\label{sec:I(G)_comp_fac}

In this section we use Theorem~\ref{fractfactors2} to prove the following equivalence, which completes the proof of Theorem~\ref{theo:equivalenz_iso_frac_component}:

\begin{theo} \label{main result_isolated_vetex_conditions}
Let $G$ be a simple graph and let $n,m$ be integers with $0<m<n$. Then
\begin{align*}
iso(G-S) \leq \frac{n}{m} \vert S \vert \quad \text{for all } S\subset V(G)
\end{align*}
if and only if $G$ has a $\{C_{2i+1},T \colon 1 \leq i < \frac{m}{n-m}, T\in\mathcal{T}_{\frac{n}{m}}\}$-factor.
\end{theo}

Observe that $\frac{m}{n-m} \leq 1$ if and only if $\frac{n}{m} \geq 2$, and hence, $\{C_{2i+1},T \colon 1 \leq i <\frac{m}{n-m}, T\in\mathcal{T}_{\frac{n}{m}}\}=\mathcal{T}_{\frac{n}{m}}$ in this case.

For two positive integers $n,m$, we say a graph $G$ satisfies the \emph{$\frac{n}{m}$-isolated-vertex-condition}, if $iso(G-S) \leq \frac{n}{m} \vert S \vert$ for all $S \subset V(G)$. To prove Theorem~\ref{main result_isolated_vetex_conditions} we need the following observation.

\begin{obs}\label{obs:isolated_vertex_condition_unconnected_graph}
A simple graph $G$ satisfies the $\frac{n}{m}$-isolated-vertex-condition, if and only if every component of $G$ satisfies the $\frac{n}{m}$-isolated-vertex-condition.
\end{obs}
\begin{proof}
If $G$ satisfies the $\frac{n}{m}$-isolated-vertex-condition and $C$ is a component of $G$, then for every $S \subset V(C)$ we have
\begin{align*}
iso(C-S) \leq iso(G-S) \leq \frac{n}{m} \vert S \vert
\end{align*}
On the other hand, if $G$ is a graph with components $H_{1},...,H_{l}$ and every component satisfies the $\frac{n}{m}$-isolated-vertex-condition, then for each $S \subset V(G)$ we have
\begin{align*}
iso(G-S) = \sum\limits_{i=1}^{l} iso \left( H_{i}-(S \cap V(H_{i})) \right) \leq \sum\limits_{i=1}^{l} \frac{n}{m} \vert S \cap V(H_{i})) \vert = \frac{n}{m} \vert S \vert.
\end{align*}
\end{proof}

For a fractional $[1,b]$-factor $h$ of a graph $G$ and $v \in V(G)$, we call $v$ a \emph{$(+)$-vertex} if $d^{h}(v)>1$ and a \emph{$(-)$-vertex} if $d^{h}(v)=1$.

\begin{proof}[Proof of Theorem~\ref{main result_isolated_vetex_conditions}.]
First, assume that $G$ has a $\{C_{2i+1},T \colon 1\leq i < \frac{m}{n-m}, T\in\mathcal{T}_{\frac{n}{m}}\}$-factor $F$. Let $H_{1},...,H_{l}$ be the components of $F$. Clearly, every component of $F$ satisfies the $\frac{n}{m}$-isolated-vertex-condition and thus, $F$ also does. For every $S \subset V(G)$ each isolated vertex of $G-S$ is also an isolated vertex of $F-S$, and thus $iso(G-S) \leq iso(F-S) \leq \frac{n}{m} \vert S \vert$.

Next, assume $G$ satisfies $iso(G-S) \leq \frac{n}{m} \vert S \vert$ for all $S \subset V(G)$. Let $F$ be an inclusion-wise minimal factor of $G$, that also satisfies the $\frac{n}{m}$-isolated-vertex-condition. By Theorem \ref{fractfactors2}, $F$ has a fractional $[1,\frac{n}{m}]$-factor, whereas every spanning proper subgraph of $F$ does not admit such a fractional factor. In particular, for every $e \in E(F)$, the graph $F-e$ does not have a fractional $[1,\frac{n}{m}]$-factor. In conclusion, the following claim holds:

\begin{claim} \label{claim:frac_factor_no_0}
$h(e) \neq 0$ for every $e \in E(F)$ and every fractional $[1,\frac{n}{m}]$-factor $h$ of $F$.
\end{claim}

We now prove that $F$ is the desired factor.

A closed trail of length $k$ (of $F$) is a sequence $(v_{0},e_{0},v_{1},e_{1},...,e_{l-1},v_{l})$ of alternately vertices and edges of $F$ with $e_{i}=v_{i}v_{i+1}$ for all $i<l$ and $v_{0}=v_{l}$.

\begin{claim} \label{claim:No_closed_even_trail}
$F$ does not contain a closed trail of an even length.
\end{claim}

\emph{Proof of Claim \ref{claim:No_closed_even_trail}.}
Suppose $F$ contains a closed trail $X$ of an even length. Let $e$ be an arbitrary edge of $X$. Now fix a fractional $[1,\frac{n}{m}]$-factor $h$ of $F$ with values in $\{0,\frac{1}{m},...,\frac{m-1}{m},1\}$, such that
\begin{itemize}
\item[$(i)$] $h(e)$ is as small as possible,
\item[$(ii)$] with respect to $(i)$, $\sum_{e' \in E(F)}h(e')$ is as small as possible.
\end{itemize}
Now suppose, there is an edge $e' \in E(F)$ between two $(+)$-vertices. By Claim~\ref{claim:frac_factor_no_0}, the edge $e'$ did not receive the value 0. Thus, reducing $h(e')$ by $\frac{1}{m}$ leads to a new fractional $[1,\frac{n}{m}]$-factor with a smaller sum, which contradicts the choice of $h$. Therefore, the set of $(+)$-vertices (with respect to $h$) is stable in $F$. This implies, that an edge of $F$ received the value 1 if and only if it is incident with a vertex of degree 1 in $F$. As a consequence, $h(e')<1$ for every edge $e'$ of $X$. Now we modify the fractional factor $h$ as follows: add $\frac{1}{m}$ and $-\frac{1}{m}$ alternately to the edges of $X$ such that $-\frac{1}{m}$ is added to $e$ (see Figure \ref{fig:1_isolated_vertex_conditions}).

\begin{figure}[htbp]
	\centering
	\scalebox{1}{\input{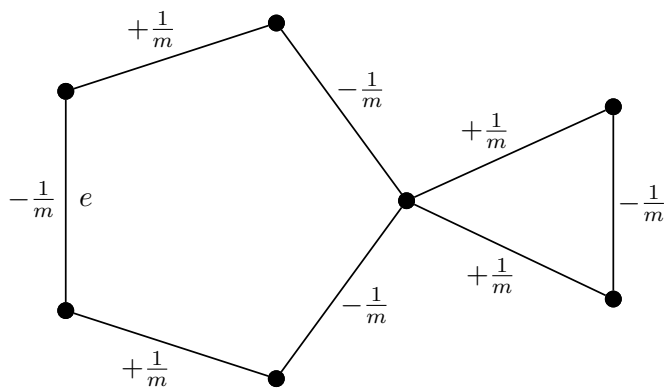}} 
	\caption{The modifying of $h$ if $F$ contains a closed trail of an even length.}
	\label{fig:1_isolated_vertex_conditions}
\end{figure}

Since no edge of $X$ had the value 0 or 1, this led to a new fractional $[1,\frac{n}{m}]$-factor $h'$ of $F$ with values in $\{0,\frac{1}{m},...,\frac{m-1}{m},1\}$. This contradicts the choice of $h$, since $h'(e)=h(e)-\frac{1}{m}$.
\ENDproof
As a consequence the following claims hold:

\begin{claim} \label{claim:No_even_circuit}
$F$ does not contain an even circuit.
\end{claim}

\begin{claim} \label{claim:No_circuits_with_common_edge}
$F$ does not contain two circuits that share an edge.
\end{claim}

\emph{Proof of Claim \ref{claim:No_circuits_with_common_edge}.}
Suppose Claim~\ref{claim:No_circuits_with_common_edge} is false. Then $F$ contains two circuits $C, C'$ such that their common edges induce a path $P$ in $F$. By Claim~\ref{claim:No_even_circuit}, the circuits $C, C'$ are odd and thus the graph induced by $E(C) \cup E(C')-E(P)$ is an even circuit. This contradicts Claim~\ref{claim:No_even_circuit}.
\ENDproof

\begin{claim} \label{claim:No_circuits_with_common_vertex}
$F$ does not contain two circuits that share a vertex.
\end{claim}

\emph{Proof of Claim \ref{claim:No_circuits_with_common_vertex}.}
Suppose $F$ contains two circuits $C, C'$ that share a vertex. By Claim~\ref{claim:No_even_circuit}, $C, C'$ are odd circuits; by Claim~\ref{claim:No_circuits_with_common_edge}, $E(C) \cap E(C')= \emptyset$. Hence, the edgeset $E(C)  \cup E(C')$ provides a closed trail of an even length, which contradicts Claim~\ref{claim:No_closed_even_trail}.
\ENDproof

\begin{claim} \label{claim:No_circuits_connected_by_path}
$F$ does not contain two disjoint circuits that are connected by a path.
\end{claim}

\emph{Proof of Claim \ref{claim:No_circuits_connected_by_path}.}
Suppose $F$ contains two disjoint circuits $C, C'$ that are connected by a path $P$. Let $e$ be an arbitrary edge of $P$. Now fix a fractional $[1,\frac{n}{m}]$-factor $h$ of $F$ with values in $\{0,\frac{1}{m},...,\frac{m-1}{m},1\}$, such that
\begin{itemize}
\item[$(i)$] $h(e)$ is as small as possible,
\item[$(ii)$] with respect to $(i)$, $\sum_{e' \in E(F)}h(e')$ is as small as possible.
\end{itemize}
Again, no two $(+)$-vertices are adjacent in $F$. This implies, $h(e')<1$ for all $e' \in E(C) \cup E(C') \cup E(P)$. Since $C$ and $C'$ are odd by Claim~\ref{claim:No_even_circuit}, both circuits contain adjacent $(-)$-vertices. In conclusion, there is a path $P'=(v_{1},...,v_{l})$ such that $E(P') \subset E(C) \cup E(C') \cup E(P)$, $e \in E(P')$ and $v_{1},v_{2}$ are two $(-)$-vertices of $C$ and $v_{l-1},v_{l}$ are two $(-)$-vertices of $C'$. Now, add $\frac{1}{m}$ and $-\frac{1}{m}$ alternately to the edges of $P'-\{v_{1}v_{2},v_{l-1}v_{l}\}$ such that $-\frac{1}{m}$ is added to $e$. If $v_{2}v_{3}$ or $v_{l-2}v_{l-1}$ received $-\frac{1}{m}$, add $\frac{1}{m}$ to $v_{1}v_{2}$ or $v_{l-1}v_{l}$, respectively. An example is shown in Figure \ref{fig:2_isolated_vertex_conditions}.

\begin{figure}[htbp]
	\centering
	\scalebox{1}{\input{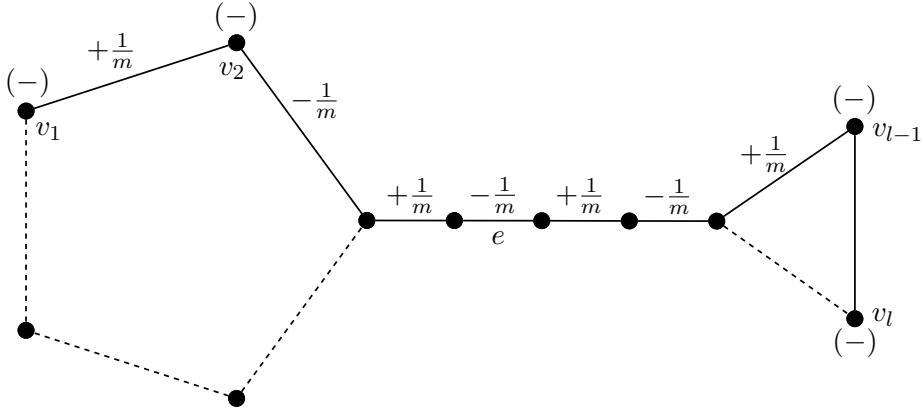}} 
	\caption{The modifying of $h$ if $F$ contains two disjoint circuits connected by a path. The solid edges are the edges of $P'$.}
	\label{fig:2_isolated_vertex_conditions}
\end{figure}

The resulting function $h'$ has values in $\{0,\frac{1}{m},...,\frac{m-1}{m},1\}$, since no edge of $C$,  $C'$ or $P$ had the value 0 or 1 before. Furthermore, we have $d^{h'}(v) \in \{d^{h}(v), d^{h}(v)+\frac{1}{m}\}$ for every $v \in \{v_{1},v_{2},v_{l-1},v_{l}\}$ and $d^{h'}(w)=d^{h}(w)$ for every other vertex $w$. Since $v_{1},v_{2},v_{l-1}$ and $v_{l}$ are $(-)$-vertices (with respect to $h$), $h'$ is a fractional $[1,\frac{n}{m}]$-factor of $F$ with values in $\{0,\frac{1}{m},...,\frac{m-1}{m},1\}$. This contradicts the choice of $h$, since $h'(e)=h(e)-\frac{1}{m}$.
\ENDproof

\begin{claim} \label{claim:No_circuit_and_deg1_vertex}
No component of $F$ contains a circuit and a vertex of degree 1.
\end{claim}

\emph{Proof of Claim \ref{claim:No_circuit_and_deg1_vertex}.}
Suppose $F$ contains a component with a circuit $C$ and a vertex $x$ with $N_{F}(x)=\{y\}$. Let $z \in N_{F}(y) \setminus \{x\}$ be a vertex such that either the edge $yz$ lies on a path from $y$ to $C$ or $y,z \in V(C)$. Now, fix a fractional $[1,\frac{n}{m}]$-factor $h$ of $F$ with values in $\{0,\frac{1}{m},...,\frac{m-1}{m},1\}$, such that
\begin{itemize}
\item[$(i)$] $h(yz)$ is as small as possible,
\item[$(ii)$] with respect to $(i)$, $\sum_{e \in E(F)}h(e)$ is as small as possible.
\end{itemize}
Again, no two $(+)$-vertices are adjacent in $F$, which implies that $C$ contains adjacent $(-)$-vertices. Furthermore, an edge received the value 1 if and only if it is incident with a vertex of degree 1, in particular $h(xy)=1$ and hence $y$ is a $(+)$-vertex. In conclusion, there is a path $P=(v_{1},...,v_{l})$ such that $v_{1}=y$, $v_{2}=z$ and $v_{l-1},v_{l}$ are two $(-)$-vertices of $C$. Now, add $\frac{1}{m}$ and $-\frac{1}{m}$ alternately to the edges of $P-v_{l-1}v_{l}$ such that $-\frac{1}{m}$ is added to $yz$. If $v_{l-2}v_{l-1}$  received $-\frac{1}{m}$, add $\frac{1}{m}$ to $v_{l-1}v_{l}$ (see Figure \ref{fig:3_isolated_vertex_conditions}).

\begin{figure}[htbp]
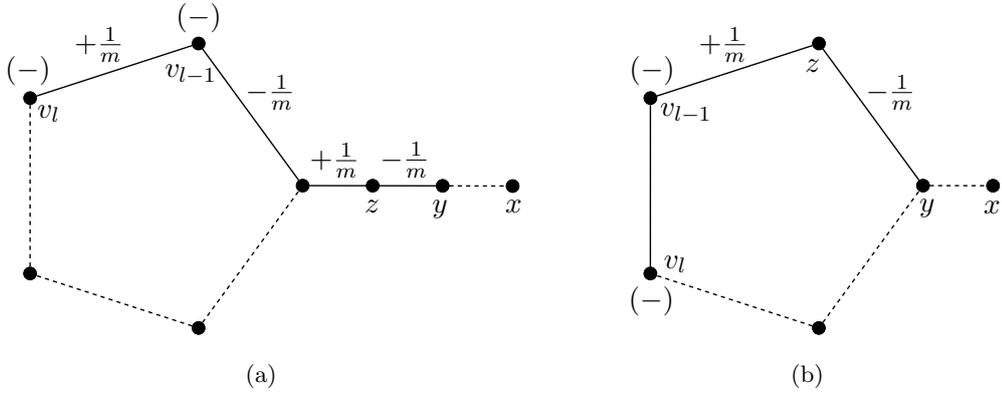

	\centering
	\subfigure[]{
		\begin{minipage}[t]{0.55\textwidth}
			\centering
			\input{circuit_leaf_1.pdf_tex}
		\end{minipage}
	}
	\subfigure[]{
		\begin{minipage}[t]{0.3\textwidth}
			\centering
			\input{circuit_leaf_2.pdf_tex}
		\end{minipage}
	}
\caption{The modifying of $h$ if $F$ contains a component with a circuit and a vertex of degree 1 in the cases (a) $y \notin V(C)$ and (b) $y \in V(C)$. The solid edges are the edges of $P$.}
\label{fig:3_isolated_vertex_conditions}
\end{figure}

The resulting function is denoted by $h'$. For each edge $e \in E(P)$ we have $h(e)>0$ by Claim~\ref{claim:frac_factor_no_0} and $h(e)<1$ since $P$ does not contain a vertex of degree 1 in $F$. In conclusion, $h'$ has values in $\{0,\frac{1}{m},...,\frac{m-1}{m},1\}$. Furthermore, we have $d^{h'}(y) = d^{h}(y)-\frac{1}{m}$, $d^{h'}(v) \in \{d^{h}(v), d^{h}(v) + \frac{1}{m}\}$ for every $v \in \{v_{l-1},v_l\}$ and $d^{h'}(w)=d^{h}(w)$ for every other vertex $w$. Since $y$ is a $(+)$-vertex and $v_{l-1}$, $v_{l}$ are $(-)$-vertices (with respect to $h$), $h'$ is a fractional $[1,\frac{n}{m}]$-factor of $F$ with values in $\{0,\frac{1}{m},...,\frac{m-1}{m},1\}$. This contradicts the choice of $h$, since $h'(yz)=h(yz)-\frac{1}{m}$.
\ENDproof

By Claims~\ref{claim:No_even_circuit}-\ref{claim:No_circuit_and_deg1_vertex}, each component of $F$ is isomorphic to either an odd circuit or a tree.

\begin{claim} \label{claim:odd_circuits_small}
If $i$ is a positive integer and $C$ is a component of $F$ isomorphic to $C_{2i+1}$, then $i<\frac{m}{n-m}$.
\end{claim}

\emph{Proof of Claim \ref{claim:odd_circuits_small}.}
By the choice of $F$ and Observation~\ref{obs:isolated_vertex_condition_unconnected_graph}, no proper subgraph of $C$ satisfies the $\frac{n}{m}$-isolated-vertex-condition. In particular, $P_{2i+1}$ does not satisfy the $\frac{n}{m}$-isolated-vertex-condition. Therefore, $\frac{i+1}{i}> \frac{n}{m}$, which is equivalent to $i<\frac{m}{n-m}$.
\ENDproof

\begin{claim} \label{claim:trees_in_Tnm}
If $T$ is a component of $F$ that is isomorphic to a tree, then $T \in \mathcal{T}_{\frac{n}{m}}$.
\end{claim}

\emph{Proof of Claim \ref{claim:trees_in_Tnm}.}
By Observation~\ref{obs:isolated_vertex_condition_unconnected_graph}, $T$ satisfies the $\frac{n}{m}$-isolated-vertex-condition, whereas no proper subgraph of $T$ satisfies this condition. Hence, $T \in \mathcal{T}_{\frac{n}{m}}$.
\ENDproof

In conclusion, every component of $F$ is isomorphic to an element of $\{C_{2i+1},T \colon 1 \leq i < \frac{m}{n-m}, T\in\mathcal{T}_{\frac{n}{m}}\}$ and thus, $F$ is the desired factor. This completes the proof of Theorem~\ref{main result_isolated_vetex_conditions}.
\end{proof}

\section{Structural properties of the trees in $\mathcal{T}_{\frac{n}{m}}$} \label{sec:Tnm}

In this section, we characterize the trees in $\mathcal{T}_{\frac{n}{m}}$ in terms of their bipartition.

\begin{theo}\label{theo:characterisation_Tnm_V2}
Let $n,m$ be integers with $0<m<n$ and let $T$ be a tree with bipartition $\{A,B\}$, where $0<|B|\leq |A|$. Then, the following statements are equivalent:

\begin{itemize}
\item[$1)$] $T \in \mathcal{T}_{\frac{n}{m}}$.

\item[$2)$] For every $x \in B$, $T$ has a fractional $[1,\frac{n}{m}]$-factor $h$ with values in $\{\frac{1}{m},...,\frac{m-1}{m},1\}$ such that $d^h(a)=1$ for every $a \in A$, $d^h(b)=\frac{n}{m}$ for every $b \in B\setminus\{x\}$ and $d^h(x)=\frac{n}{m}+|A|-\frac{n}{m}|B|$.

\item[$3)$] $|A| \leq \frac{n}{m} |B|$ and for every $e=xy \in E(T)$: $|V(T_e) \cap A|>\frac{n}{m} |V(T_e)\cap B|$, where $T_e$ is the component of $T-e$ that contains the unique vertex in $\{x,y\} \cap A$.

\end{itemize}
\end{theo}
\begin{proof}
$1) \Rightarrow 2)$. For stars $2)$ trivially holds. Thus, we assume $T$ is not a star and hence, there is an $u \in Leaf(T-Leaf(T))$. Recall that no fractional $[1,\frac{n}{m}]$-factor of $T$ uses value $0$. Let $h$ be a fractional $[1,\frac{n}{m}]$-factor of $T$ with values in $\{\frac{1}{m},...,\frac{m-1}{m},1\}$, such that
\begin{itemize}
\item[$(i)$] $d^h(u)$ is as small as possible,
\item[$(ii)$] with respect to $(i)$, $\sum_{e \in E(T)}h(e)$ is as small as possible.
\end{itemize}
Observe that no two $(+)$-vertices are adjacent and as a consequence, $h(e)=1$ if and only if $e$ is a pendant edge of $T$. Furthermore, every vertex adjacent to a leaf of $T$ is a $(+)$-vertex since $T$ is not isomorphic to $K_2$.

First, suppose $T$ contains a path $P=(u,v_1, \ldots, v_l)$ in $T$ such that $v_{l-1}$ is a $(-)$-vertex and $d^h(v_l)<\frac{n}{m}$. Modify $h$ as follows: add $-\frac{1}{m}$ and $\frac{1}{m}$ alternately to the edges of $P-v_{l-1}v_{l}$ such that $-\frac{1}{m}$ is added to $uv_1$. If $v_{l-2}v_{l-1}$ received $-\frac{1}{m}$, add $\frac{1}{m}$ to $v_{l-1}v_{l}$, see Figure~\ref{fig:4_isolated_vertex_conditionsV2}.

\begin{figure}[htbp]
	\centering
	\scalebox{1}{\input{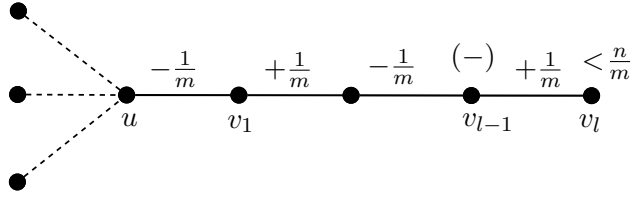}} 
	\caption{The modifying of $h$ if $T$ contains a path $P=(u,v_1, \ldots, v_l)$ such that $v_{l-1}$ is a $(-)$-vertex and $d^h(v_l)<\frac{n}{m}$. The solid edges belong to $P$.}
	\label{fig:4_isolated_vertex_conditionsV2}
\end{figure}

Note that $v_l$ is not a leaf, since it is adjacent to a $(-)$-vertex. Hence, no edge of $P$ is a pendant edge of $T$ and thus, every $e \in E(P)$ satisfies $h(e)<1$. In conclusion, the modification of $h$, denoted by $h'$, has values in $\{0,\frac{1}{m},...,\frac{m-1}{m},1\}$. Moreover, $h'$ is a fractional $[1,\frac{n}{m}]$-factor of $T$ since $d^h(u)>1$, $d^h(v_{l-1})=1$ and $d^h(v_l)<\frac{n}{m}$. This contradicts the choice of $h$, since $d^{h'}(u)=d^h(u)-\frac{1}{m}$.

The non-existence of such a path implies that the set of $(-)$-vertices is stable and every $v \in V(T)\setminus \{u\}$ that is a $(+)$-vertex satisfies $d^h(v)= \frac{n}{m}$. The former implies that $A$ consists of all $(-)$-vertices and $B$ of all $(+)$-vertices. Hence,
\begin{align*}
|A|=\sum\limits_{a \in A} d^h(a)=\sum\limits_{b \in B}d^h(b)=\frac{n}{m}(|B|-1)+d^h(u),
\end{align*}
which implies $d^h(u)=\frac{n}{m}+|A|-\frac{n}{m}|B|$.

Now, let $x$ be an arbitrary $(+)$-vertex, let $P$ be the $ux$-path contained in $T$ and let $l=m\left(\frac{n}{m}-d^h(u)\right)$. Note that $|V(P)|$ is odd, since $P$ consists of alternately $(+)$- and $(-)$-vertices. Set $h_0=h$ and for $i \in \{1,\ldots,l\}$ let $h_i$ be the function obtained from $h_{i-1}$ by alternately adding $\frac{1}{m}$ and $-\frac{1}{m}$ to the edges of $P$ such that $\frac{1}{m}$ is added to the edge of $P$ incident with $u$ (see Figure~\ref{fig:5_isolated_vertex_conditions}).

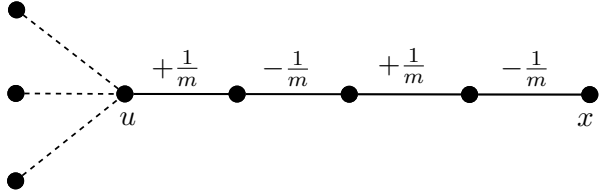
\begin{figure}[htbp]
	\centering
	\scalebox{1}{
\begingroup%
  \makeatletter%
  \providecommand\color[2][]{%
    \errmessage{(Inkscape) Color is used for the text in Inkscape, but the package 'color.sty' is not loaded}%
    \renewcommand\color[2][]{}%
  }%
  \providecommand\transparent[1]{%
    \errmessage{(Inkscape) Transparency is used (non-zero) for the text in Inkscape, but the package 'transparent.sty' is not loaded}%
    \renewcommand\transparent[1]{}%
  }%
  \providecommand\rotatebox[2]{#2}%
  \newcommand*\fsize{\dimexpr\f@size pt\relax}%
  \newcommand*\lineheight[1]{\fontsize{\fsize}{#1\fsize}\selectfont}%
  \ifx\svgwidth\undefined%
    \setlength{\unitlength}{222.23394458bp}%
    \ifx\svgscale\undefined%
      \relax%
    \else%
      \setlength{\unitlength}{\unitlength * \real{\svgscale}}%
    \fi%
  \else%
    \setlength{\unitlength}{\svgwidth}%
  \fi%
  \global\let\svgwidth\undefined%
  \global\let\svgscale\undefined%
  \makeatother%
  \begin{picture}(1,0.31778778)%
    \lineheight{1}%
    \setlength\tabcolsep{0pt}%
    \put(0,0){\includegraphics[width=\unitlength,page=1]{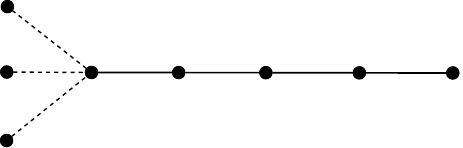}}%
    \put(0.24207246,0.19240997){\makebox(0,0)[lt]{\lineheight{1.25}\smash{\begin{tabular}[t]{l}$+\frac{1}{m}$\end{tabular}}}}%
    \put(0.42763453,0.19106304){\makebox(0,0)[lt]{\lineheight{1.25}\smash{\begin{tabular}[t]{l}$-\frac{1}{m}$\end{tabular}}}}%
    \put(0.82819223,0.19141217){\makebox(0,0)[lt]{\lineheight{1.25}\smash{\begin{tabular}[t]{l}$-\frac{1}{m}$\end{tabular}}}}%
    \put(0.62210883,0.19646565){\makebox(0,0)[lt]{\lineheight{1.25}\smash{\begin{tabular}[t]{l}$+\frac{1}{m}$\end{tabular}}}}%
    \put(0.95671791,0.10732478){\makebox(0,0)[lt]{\lineheight{1.25}\smash{\begin{tabular}[t]{l}$x$\end{tabular}}}}%
    \put(0.1885875,0.10945832){\makebox(0,0)[lt]{\lineheight{1.25}\smash{\begin{tabular}[t]{l}$u$\end{tabular}}}}%
  \end{picture}%
\endgroup%
} 
	\caption{The modifying of $h_{i-1}$ to obtain $h_i$. The solid edges belong to $P$.}
	\label{fig:5_isolated_vertex_conditions}
\end{figure}

We have $d^{h_l}(u)=d^h(u)+\frac{l}{m}=\frac{n}{m}$ and $d^{h_l}(x)=d^h(x)-\frac{l}{m}=\frac{n}{m}-\frac{l}{m}=d^h(u)$. As a consequence, $d^{h_i}(v) \in [1,\frac{n}{m}]$ for every $i \in \{1,\ldots,l\}$ and every $v \in V(T)$. Furthermore, for every $i \in \{1,\ldots,l\}$, if $h_{i-1}$ is a fractional $[1,\frac{n}{m}]$-factor that does not use value $0$ nor $1$ on $P$, then $h_i$ is a fractional $[1,\frac{n}{m}]$-factor. Thus, $h_i$ also does not use value $0$ on $P$. Moreover, it also does not use value $1$ on $P$, since every edge of $P$ is incident with a $(-)$-vertex (with respect to $h_i$) that is not a leaf of $T$. As a consequence, for every $i \in \{1, \ldots,l\}$, $h_i$ only uses values $\frac{1}{m},\ldots,\frac{m-1}{m}$ on $P$ and therefore, $h_l$ is the desired fractional factor.

$2) \Rightarrow 3)$ By Theorem~\ref{fractfactors2}, $T$ satisfies the $\frac{n}{m}$-isolated-vertex-condition and hence $|A|=iso(T-B) \leq \frac{n}{m}|B|$. Let $e=xy \in E(T)$, where $y \in A$, and let $h$ be a fractional $[1,\frac{n}{m}]$-factor of $T$ with the properties stated in $2)$ (with predescribed vertex $x$). Then,
\begin{align*}
|V(T_e) \cap A|=\sum\limits_{v \in V(T_e) \cap A}d^h(v)=h(xy)+\sum\limits_{w \in V(T_e) \cap B}d^h(w)=h(xy)+\frac{n}{m} |V(T_e)\cap B|,
\end{align*}
which proves $3)$, since $h(xy)>0$.

$3) \Rightarrow 1)$ For every $e \in E(T)$, statement $3)$ implies
\begin{align*}
iso\left(\left(T-e\right)-\left(V(T_e) \cap B\right)\right)=|V(T_e) \cap A|>\frac{n}{m} |V(T_e)\cap B|.
\end{align*}
Thus, by Theorem~\ref{fractfactors2} it suffices to show that $T$ has a fractional $[1,\frac{n}{m}]$-factor. For every $e \in E(T)$ set
\begin{align*}
h(e)=|V(T_e) \cap A|-\frac{|A|}{|B|}|V(T_e) \cap B|.
\end{align*}
For every $e \in E(T)$, statement $3)$ implies
\begin{align*}
h(e)=|V(T_e) \cap A|-\frac{|A|}{|B|}|V(T_e) \cap B| \geq |V(T_e) \cap A|-\frac{n}{m}|V(T_e) \cap B| >0.
\end{align*}
By the definition of $h$, for every $a \in A$ and every $b \in B$ we have
\begin{align*}
d^h(a) &= \sum\limits_{e' \in \partial_T(a)}h(e')=(d_T(a)-1)(|A|-1) + d_T(a)- \frac{|A|}{|B|}(d_T(a)-1)|B| \\&=  d_T(a)|A|-|A|+1 - |A|(d_T(a)-1) =1
\end{align*}
and
\begin{align*}
d^h(b) &= \sum\limits_{e' \in \partial_T(b)}h(e')=|A| - \frac{|A|}{|B|}\left(|B|-1\right) = \frac{|A|}{|B|}.
\end{align*}
Note that $1< \frac{|A|}{|B|}\leq \frac{n}{m}$, since $|A|>|B|$. Furthermore, for every $e=xy \in E(T)$, where $x \in B$, the above calculations imply
\begin{align*}
h(e)=d^h(y)-\sum\limits_{e' \in \partial_T(y)\setminus\{e\}}h(e')\leq 1.
\end{align*}
In conclusion, $h$ is a fractional $[1,\frac{n}{m}]$-factor of $T$, which proves $T \in \mathcal{T}_{\frac{n}{m}}$.
\end{proof}

Note that, by the proof of $3) \Rightarrow 1)$, every $T \in \mathcal{T}_{\frac{n}{m}}$ has a fractional $[1,\frac{n}{m}]$-factor $h$ such that $d^h(a)=1$ for every $a \in A$ and $d^h(b)=\frac{|A|}{|B|}$ for every $b \in B$. On the other hand, not every tree with such a factor belongs to $\mathcal{T}_{\frac{n}{m}}$. As the following corollary shows, Theorem~\ref{theo:characterisation_Tnm_V2} implies some structural properties of trees in $\mathcal{T}_{\frac{n}{m}}$.

\begin{cor}
Let $n,m$ be integers with $0<m<n$ and let $T \in \mathcal{T}_{\frac{n}{m}}$ be a tree with bipartition $\{A,B\}$, where $0<|B|\leq |A|$. Then, the following holds
\begin{itemize}
\item[$(i)$] either $T \cong K_{1,1}$, or $Leaf(T) \subseteq A$,
\item[$(ii)$] $d_T(a) \leq m$ for every $a \in A$,
\item[$(iii)$] $d_T(b) \leq n$ for every $b \in B$,
\item[$(iv)$] $d_T(x) = \lfloor \frac{n}{m} \rfloor+1$ for every $x \in Leaf(T-Leaf(T))$,
\item[$(v)$] if $n \equiv 1 \text{ (mod m)}$, then either $T$ is a star or $|A| =\frac{n}{m}|B|$ and $|V(T)|$ is a multiple of $n+m$.
\end{itemize}
\end{cor}

\begin{proof}
For stars the statements are trivial. Thus, assume $T$ is not a star and hence, there are two distinct vertices $x_1,x_2 \in Leaf(T-Leaf(T))$. Note that every vertex $v \in Leaf(T-Leaf(T))$ belongs to $B$, since $h(v)>1$ for every fractional $[1,\frac{n}{m}]$-factor $h$ of $T$. Consider two fractional $[1,\frac{n}{m}]$-factors $h_1, h_2$ of $T$ with the properties stated in statement $2)$ of Theorem~\ref{theo:characterisation_Tnm_V2} (with respect to $x_1$ and $x_2$, respectively). The existence of $h_1$ implies $(i), (ii), (iii)$ and $d(x) = \lfloor \frac{n}{m} \rfloor+1$ for every $x \in Leaf(T-Leaf(T))\setminus \{x_1\}$. By the existence of $h_2$ we have $d(x_1) = \lfloor \frac{n}{m} \rfloor+1$, which proves $(iv)$. Furthermore, if $n \equiv 1 \text{ (mod m)}$, then $\frac{n}{m} =  \lfloor \frac{n}{m} \rfloor+\frac{1}{m} \leq d^{h_1}(x_1) \leq \frac{n}{m}$. Hence, $\frac{n}{m}=d^{h_1}(x_1)=\frac{n}{m}+|A|-\frac{n}{m}|B|$, i.e. $|A| =\frac{n}{m}|B|$. Moreover, we observe that $\frac{|B|}{m}$ is an integer, since $\frac{n-1}{m}$ is an integer and $|A|=\frac{n}{m}|B|=\frac{n-1}{m}|B|+\frac{|B|}{m}$. As a consequence, $|V(G)|=|A|+|B|=\frac{n}{m}|B|+|B|=(n+m)\frac{|B|}{m}$, which proves $(v)$.
\end{proof}

By $(i), (iii)$ and $(iv)$, for every $T \in \mathcal{T}_{\frac{n}{1}}$, the set $Leaf(T-Leaf(T))$ is empty, which is equivalent to $T$ being a star. As a consequence, $\mathcal{T}_{\frac{n}{1}}=\{K_{1,i} \colon 1 \leq i \leq n\}$, or equivalently, Theorem~\ref{theo:star_factor} holds.

\bibliography{Lit_iso_toughness}{}

\begin{thebibliography}{10}

\bibitem{factors_and_factorizations_book}
J.~Akiyama and M.~Kano.
\newblock {\em Factors and factorizations of graphs: {P}roof techniques in
  factor theory}.
\newblock Springer-Verlag, Berlin, Heidelberg, 1st edition, 2011.

\bibitem{amahashi1982factors}
A.~Amahashi and M.~Kano.
\newblock On factors with given components.
\newblock {\em Discrete Mathematics}, 42(1):1--6, 1982.

\bibitem{ANSTEE199029}
R.~Anstee.
\newblock Simplified existence theorems for $(g,f)$-factors.
\newblock {\em Discrete Applied Mathematics}, 27(1):29--38, 1990.

\bibitem{heinrich_et_al_1990}
K.~Heinrich, P.~Hell, D.~Kirkpatrick, and G.~Liu.
\newblock A simple existence criterion for $(g<f)$-factors.
\newblock {\em Discrete Mathematics}, 85:313--317, 1990.

\bibitem{kano2010component}
M.~Kano, H.~Lu, and Q.~Yu.
\newblock Component factors with large components in graphs.
\newblock {\em Applied Mathematics Letters}, 23(4):385--389, 2010.

\bibitem{las1978extension}
M.~Las~Vergnas.
\newblock An extension of {T}utte's 1-factor theorem.
\newblock {\em Discrete Mathematics}, 23(3):241--255, 1978.

\bibitem{ma2009isolated}
Y.~Ma, A.~Wang, and J.~Li.
\newblock Isolated toughness and fractional $(g,f)$-factors of graphs.
\newblock {\em Ars Combinatoria}, 93:153--160, 2009.

\bibitem{tutte19531}
W.~T. Tutte.
\newblock The 1-factors of oriented graphs.
\newblock {\em Proceedings of the American Mathematical Society},
  4(6):922--931, 1953.

\bibitem{yang2001}
J.~Yang, Y.~Ma, and G.~Liu.
\newblock Fractional {{\((g,f)\)}}-factors in graphs.
\newblock {\em Applied Mathematics, Series A (Chinese Edition)},
  16(4):385--390, 2001.

\bibitem{Yu2019FractionalFC}
R.~Yu, M.~Kano, and H.~Lu.
\newblock Fractional factors, component factors and isolated vertex conditions
  in graphs.
\newblock {\em Electronic Journal of Combinatorics}, 26:4, 2019.

\end{thebibliography}
\addcontentsline{toc}{section}{References}
\bibliographystyle{abbrv}

\end{document}